\documentclass{amsart}

\usepackage{graphicx}
\usepackage{amssymb}
\usepackage{xcolor}

\newtheorem{theorem}{Theorem}[section]

\newtheorem{proposition}[theorem]{Proposition}
\newtheorem{corollary}[theorem]{Corollary}

\theoremstyle{definition}
\newtheorem{definition}[theorem]{Definition}
\newtheorem{example}[theorem]{Example}

\theoremstyle{remark}
\newtheorem{remark}[theorem]{Remark}

\DeclareMathOperator{\rank}{rk}
\numberwithin{equation}{section}

\begin{document}

\title[Flows that have a non-saddle or a $W$-set]{Topology and dynamics of a flow that has a non-saddle set or a $W$-set}

%    Only \author and \address are required; other information is
%    optional.  Remove any unused author tags.

%    author one information
% \author[short version for running head]{name for top of paper}
\author{H. Barge}
\address{E.T.S. Ingenieros Informáticos, Universidad Politécnica de Madrid, 28660 Madrid, Spain}
\email{h.barge@upm.es}

\author{J.J. Sánchez-Gabites}
\address{Facultad de Ciencias Matemáticas, Universidad Complutense de Madrid, 28040 Madrid, Spain}
\curraddr{}
\email{jajsanch@ucm.es}
%    author two information
\author{J.M.R. Sanjurjo}
\address{Facultad de Ciencias Matemáticas, Universidad Complutense de Madrid, 28040 Madrid, Spain}
\email{jose\_sanjurjo@mat.ucm.es}

\thanks{The authors are funded by the Spanish Ministerio de Ciencia e Innovación through grant PGC2018-098321-B-I00 and, for the second author, also grant RYC2018-025843-I}

%    \subjclass is required.

\date{}

\dedicatory{}

%    Abstract is required.
\begin{abstract} The aim of this paper is to study dynamical and topological properties of a flow in the region of influence of an isolated non-saddle set or a $W$-set in a manifold. These are certain classes of compact invariant sets in whose vicinity the asymptotic behaviour of the flow is somewhat controlled. We are mainly concerned with global properties of the dynamics and establish cohomological relations between the non-saddle set and the manifold. As a consequence we obtain a dynamical classification of surfaces (orientable and non-orientable). We also examine robustness and bifurcation properties of non-saddle-sets and study in detail the behavior of $W$-sets in 2-manifolds.
\end{abstract}

\maketitle

\section{Introduction}

We study in the paper topological aspects of the theory of non-saddle sets and $W$-sets and of their region of influence. We define these classes of compact invariant sets in the next paragraphs. Their common feature is that the long term behaviour of the flow in their vicinity is controlled to a certain extent. Thus, in some respects, the study of this kind of sets can be seen as a general theory of stability and attraction, which extends the classical one and includes recent developments such as the theory of unstable attractors with no external explosions. We are mainly concerned with global properties of the dynamics.

The following definition is due to Bhatia \cite{BhJDE}. A compact invariant set $K$ is called \emph{non-saddle} if every neighbourhood $U$ of $K$ contains another neighbourhood $V$ of $K$ such that for every $x \in V$ either $x \cdot [0,+\infty) \subseteq U$ or $x \cdot (-\infty,0] \subseteq U$. Thus non-saddle sets exhibit a sort of ``mixed'' Lyapunov stability in the sense that trajectories of points that are sufficiently close to $K$ remain close to $K$ either in positive or negative time. Clearly a stable attractor is a non-saddle set, but certain mildly unstable attractors are also non-saddle sets, and the same holds in reverse time (repellers). We will always consider non-saddle sets $K$ that are isolated in the sense of Conley; i.e. $K$ is the largest invariant set in some neighbourhood of itself. Then the flow in the vicinity of $K$ can be decomposed into two disjoint parts: one comprised of points whose trajectories remain close and converge asymptotically to $K$ in positive time and the other comprised of points that satisfy the dual phenomenon in negative time. In other words, $K$ separates $M$ locally into components that are either attracted or repelled by $K$.

A compact invariant set $K$ is called a \emph{$W$-set} if it has a neighbourhood $U$ with the property that the orbit of every point in $U$ passes arbitrarily close to $K$; that is, for every $x \in U$ and every neighbourhood $V$ of $K$ there exists $t$ such that $x \cdot t \in V$. For example, an (isolated) non-saddle set is a $W$-set, and so is any attractor (whether stable or not) and, more generally, the union of the Morse sets of a Morse decomposition of $M$. The ``$W$'' in ``$W$-set'' stands for ``witness'', since an observer sitting at $K$ would observe $x$ passing infinitely many times, and arbitrarily close to, $K$.

Both (isolated) non-saddle sets and $W$-sets capture, to some extent, the long term behaviour of nearby points. In the case of non-saddle sets, they contain the $\omega$- or $\omega^*$-limit of every nearby point. In the case of $W$-sets, they satisfy the much weaker condition that they intersect the $\omega$- or $\omega^*$-limit of every nearby point. The set of points in $M$ that are dynamically connected to $K$ in this fashion is, in both cases, an open and invariant set which we think of as being ``large'' as opposed to the ``small'' set $K$. We call it the \emph{region of influence} (for non-saddle sets) and the \emph{witness region} (for $W$-sets) of the $K$, respectively.

In Sections \ref{sec:nonsaddle1} and \ref{sec:nonsaddle2} we obtain relations between the cohomology of a non-saddle set and global properties of the dynamics. For instance, Theorem \ref{thm:exact} provides an exact sequence relating the cohomology of a global non-saddle set to that of the manifold and a section of the flow. An interesting consequence of this result is a dynamical classification of surfaces (orientable and non-orientable), obtained in Corollary \ref{cor:classification}. In Section \ref{sec:bifurcations} we study bifurcation properties resulting from a change of stability of non-saddle sets. This phenomenon can lead to the appearance of families of new attractors or repellers in the complement of the non-saddle set and converging to it.
 
Sections \ref{sec:Wsets1} and \ref{sec:Wsets2} are devoted to $W$-sets. We obtain some relationships for this new class that generalize the classical Morse inequalities for flows defined on 2-dimensional manifolds. We prove, for example, that when the phase space is a manifold and the $W$-set is global then the Euler characteristics of the manifold agrees with that of the Conley index of the $W$-set. We also give some sufficient (and in some cases also necessary) conditions for a $W$-set to be non-saddle. An interesting consequence of our study of $W$-sets in 2-manifolds is the determination of a characteristic property for a non-saddle set not to have dissonant points.

\section{Background} \label{sec:background}

We consider a continuous flow $\varphi : M \times \mathbb{R} \longrightarrow M$ on a boundariless manifold $M$. For our purposes we will assume without loss of generality that $M$ is connected. 
 We abbreviate $\varphi(x,t)$ as $x t$. A set $K$ is invariant if $K \mathbb{R} = K$. In order to avoid trivial cases we will usually take invariant sets to satisfy $\emptyset \neq K \subsetneq M$ unless explicitly stated otherwise.

\subsection{Attractors} An attractor (or a stable attractor, for emphasis) is a compact invariant set $K$ which satisfies the following two conditions: (i) it has a neighbourhood $U$ such that the $\omega$-limit of any $x \in U$ is nonempty and contained in $K$; (ii) it is (positively) stable in the sense of Lyapunov; i.e. every neighbourhood $U$ of $K$ contains another neighbourhood $V$ of $K$ such that $V  [0,+\infty) \subseteq U$. When a compact invariant set satisfies (i) but not necessarily (ii) we will call it an unstable attractor. In either case the largest set $U$ that satisfies (i) is called the basin of attraction of $K$ and denoted by $\mathcal{A}(K)$. It is open and invariant. By time-reversing conditions (i) and (ii) one obtains the notion of a repeller $K$ and its basin of repulsion $\mathcal{R}(K)$.

\subsection{Isolated invariant sets} Let $N$ be a compact subset of phase space $M$. Its maximal invariant subset is the largest invariant set it contains; i.e. $\{x \in N : x (-\infty,+\infty) \subset N\}$. We will also consider the sets \[N^+ := \{x \in N : x [0,+\infty) \subset N\} \quad \text{ and } \quad N^- := \{x \in N : x (-\infty,0] \subset N\},\] which are the largest positively (resp. negatively) invariant subsets of $N$. When nonempty, all these sets are compact.

A compact invariant set $K$ is called isolated (in the sense of Conley \cite{Con}) when it is the maximal invariant subset of some neighbourhood $N$ of itself, called an isolating neighbourhood of $K$. Clearly $K = N^+ \cap N^-$. Of course, an isolated invariant set $K$ usually has many different isolating neighbourhoods $N$, and it is convenient to work with ones that have some nice properties, mainly related to how the flow intersects the boundary of $N$. We describe this next.

\subsection{Isolating blocks} A point $x$ in the boundary (the topological frontier) $\partial N$ is called an immediate exit point if $x  (0,\epsilon) \not\subset N$ for every $\epsilon > 0$ and an immediate entry point if $x (-\epsilon,0) \not\subset N$ for every $\epsilon > 0$. The set of immediate exit points is denoted by $N^o$; the set of immediate entry points by $N^i$. For the purposes of this paper we will need conditions much more stringent than these. We say that $x$ is a transverse entry point (into $N$) if there exists $\epsilon > 0$ such that $x (-\epsilon,0) \cap N = \emptyset$ and $x (0,\epsilon) \subset {\rm int}\ N$; a transverse exit point if there exists $\epsilon > 0$ such that $x  (-\epsilon,0) \subset {\rm int}\ N$ and $x (0,\epsilon) \cap N = \emptyset$, and an exterior tangency point if there exists $\epsilon > 0$ such that $x (-\epsilon,0)$ and $x  (0,\epsilon)$ are both disjoint from $N$.

Suppose that every point in $\partial N$ is of one of the three types just described. One easily sees that $N^o$ is comprised of the transverse exit and (exterior) tangency points, and similarly $N^i$ consists of the transverse entry and exterior tangency points. If additionally both $N^o$ and $N^i$ are closed, then $N$ is called an \emph{isolating block}. (We warn the reader that this terminology is not standardized across the literature). Every isolated invariant set $K$ in $M$ admits an isolating neighbourhood $N$ which is an isolating block (\cite{chjde}). Also, in dimensions $n = 2,3$ every isolated invariant set admits a basis of isolating blocks $N$ with the additional property that $N$ is a compact $n$-manifold with boundary $\partial N$ and $N^i$ and $N^o$ are compact $(n-1)$--manifolds which intersect along their boundary. For smooth flows this is actually true in all dimensions (\cite{ConEast}).

The neat behaviour of a flow on the boundary of an isolating block $N$ is very useful. The ``exit time'' map $t^o : N \longrightarrow [0,+\infty]$ is defined by \[t^o(x) := \sup \{t \in [0,+\infty) : x  [0,t] \subset N\}\] and it is continuous. Notice that if a point $x \in N$ exits $N$ in forward time (i.e. $t^o(x) < +\infty$), it must do so through a point in $N^o$; moreover, it will exit $N$ the first time it hits $N^o$. The whole trajectory segment $x \cdot (0,t^o(x))$ is contained in the interior of $N$.

\subsection{Homotopies induced by the flow} In the paper we will often use the flow to construct homotopies and ``infinite time'' homotopies (more precisely, shape equivalences). Since this is a standard technique we will not give details and limit ourselves to illustrating it here with an example. Let $N$ be an isolating block for an isolated invariant set $K$ and consider its subset $N^+$. This is clearly positively invariant, and the flow $\varphi$ restricts to a semiflow $\overline{\varphi}$ on it. Define the sets $P_k \subseteq N^+$ to be the image of $N^+$ under the semiflow $\overline{\varphi}$ at times $k = 0,1,2,\ldots$ These $P_k$ form a decreasing sequence of compact sets which starts at $P_0 = N^+$ and whose intersection is $K$. Moreover each inclusion $P_{k+1} \subseteq P_k$ is a homotopy equivalence, with the flow providing a homotopy inverse. Thus by the continuity property of \v{C}ech cohomology one has $\check{H}^*(N^+) \cong \check{H}^*(K)$, where the isomorphism is induced by the inclusion $K \subseteq P$. Intuitively, $\overline{\varphi}$ provides a deformation retraction of $N^+$ onto $K$ in ``infinite time'', which is why the end result is an isomorphism in \v{C}ech cohomology (but perhaps not in singular cohomology). A detailed account of \v{C}ech cohomology can be found in the book by Spanier \cite{Span}. Besides continuity, we will also need the strong excision property of \v{C}ech cohomology: for any compact pair $(A,B)$ one has $\check{H}^*(A,B) = \check{H}^*(A/B,[B])$. These properties are true regardless of the coefficient group taken to compute the cohomology.

%The same idea can be applied, for example, to the semiflow $\overline{\varphi}$ induced by $\varphi$ in the quotient $N/N^o$. The $P_k$ are defined in the same way as before and now their intersection is the image of $N^-$ in the quotient space $N/N^o$; i.e. $N^-/n^-$ (INTRODUCIR LA NOTACION). Hence $\check{H}^*(N/N^o,[N^o]) \cong \check{H}^*(N^-/n^-,[n^-])$.

\subsection{The Conley index} A convenient way to describe the dynamics near an isolated invariant set $K$ is provided by the Conley index $h(K)$ (\cite{Con}). We recall briefly the definition. An index pair for $K$ is a compact pair $(N,L)$ such that:
\begin{itemize}
    \item $\overline{N \setminus L}$ is an isolating neighbourhood of $K$.
    \item $L$ is positively invariant in $N$: if $x \in L$ and $x [0,t] \subseteq N$, then $x  [0,t] \subset L$.
    \item For every $x \in N$, if $x [0,+\infty) \not\subset N$ then there exists $t \geq 0$ such that $x t \in L$ (i.e. if the forward orbit of a point $x \in N$ exits $N$, it must do so through $L$).
\end{itemize}

For example, if $N$ is an isolating block then $(N,N^o)$ is an index pair. The Conley index of $K$ is defined as the homotopy type $h(K)$ of the pointed space $(N/L,[L])$. The foundational theorem about the Conley index is that $h(K)$ is independent of the index pair $(N,L)$. One usually works with the (co)homological Conley index, which is just the (co)homology of $h(K)$, or even with the Conley-Euler characteristic, which is the Euler characteristic of $h(K)$.

%Suppose the phase space is a boundariless $2$-manifold. We can then take an isolating block $N$ which is a compact $2$-manifold, with $N^o$ a compact submanifold of $\partial N$. Since $n^-$ is a compact subset of the $1$--manifold $\partial N$, its connected components are points, arcs, and at most finitely many copies of $\mathbb{S}^1$. Thus the cohomology of $n^-$ is finitely generated except possibly in degree zero, and $\chi(n^-)$ is the number of points or arcs in $n^-$. The cohomological Conley index of $K$ is \[H^*(N/N^o,[N^o]) = H^*(N,N^o) = \check{H}^*(N^-,n^-).\] These groups are finitely generated because $N$ and $N^o$ are compact manifolds. It follows from the long exact sequence for the pair $(N^-,n^-)$ that $\check{H}^0(N^-)$ is finitely generated if and only if so is $\check{H}^0(n^-)$, and $\check{H}^k(N^-)$ is finitely generated for $k =1,2$ (and vanishes for larger $k$). In particular $\chi(N^-)$ is always defined and it is $+\infty$ precisely when $N^-$ has infinitely many connected components. The same holds true for $K$. Finally, the Euler characteristic of the Conley index is then \[\chi h(K) = \chi (N,N^o) = \chi(N^-,n^-) = \chi(N^-) - \chi(n^-) = \chi(K) - \chi(n^-).\] Since $\chi(n^-)$ is the number of points or arcs in $n^-$, it follows that $\chi h(K) \leq \chi(K)$ and the equality holds if and only if $\chi(n^-) = 0$, i.e. $n^-$ consists a some copies of $\mathbb{S}^1$, which entails that it is a union of components of $\partial N$. 

\section{The structure of a flow having a global non-saddle set}\label{sec:nonsaddle1}

Recall from the Introduction that a non-saddle set $K$ is a compact invariant set such that every neighbourhood $U$ of $K$ contains a neighbourhood $V$ of $K$ with the property that every $x \in V$ has its positive or its negative semiorbit contained in $U$. If a non-saddle set $K$ is isolated then it possesses a basis of isolating blocks that are of the form $N=N^+\cup N^-$ (see for instance \cite[Proposition~3]{BSdis}). In fact, this property characterizes when an isolated invariant set is non-saddle. In this paper we only deal with flows defined on manifolds and, hence, isolating blocks of the form $N=N^+\cup N^-$ are ANRs. As a consequence, their \v Cech homology and cohomology groups coincide with their singular homology and cohomology groups and are finitely generately and zero for almost every dimension. Moreover, using homotopies provided by the flow, it can be proved that the inclusion $K \subset N$ induces isomorphisms $\check{H}^*(N)\cong\check{H}^*(K)$ and, therefore, $K$ is also of finite type (see the proof of \cite[Theorem~4]{GMRSo} for a detailed account of the construction). In particular, $K$ has a finite number of connected components and all of them are isolated non-saddle sets. We shall also make use of the fact that if $K$ is the union of finitely many isolated non-saddle sets, $K$ is also an isolated non-saddle set.

Given an isolated non-saddle set $K$, and similar to the case of attractors and repellers, it makes sense to consider the set of points $\mathcal{I}(K)$ that are \emph{influenced} by $K$, that is, points whose $\omega$ or $\omega^*$-limit is non-empty and contained in $K$. This neighbourhood of $K$ is called its \emph{region of influence} and turns out to be open and invariant (see \cite[Proposition~5]{BSdis}). While the basin of attraction of an attractor and the region of influence of an isolated non-saddle set have many common features, the structure of the flow in the latter may be more involved. This structure has been thoroughly explored in \cite{BSdis}. We shall describe it briefly since it will play a role along the paper. If $K$ is an isolated non-saddle set, $N$ is an isolating block of the form $N=N^+\cup N^-$ and $\mathcal{I}(K)$ is its region of influence, then $\mathcal{I}(K)\setminus K$ has a finite number of components. Each component $C$ of $\mathcal{I}(K)\setminus K$ can be of one of the following four types:
\begin{itemize}
    \item All the points of $C$ are attracted by $K$ and $\omega^*(x)\cap K=\emptyset$. The flow $\varphi_{|C}$ is parallelizable with a component of $N^i$ as a global section.
    \item All the points of $C$ are repelled by $K$ and $\omega(x)\cap K=\emptyset$. The flow $\varphi_{|C}$ is parallelizable with a component of $N^o$ as a global section.
    \item All the points of $C$ are homoclinic, that is $\emptyset\neq \omega(x)\subset K$ and $\emptyset\neq \omega^*(x)\subset K$. The flow $\varphi_{|C}$ is parallelizable with a component of $N^o$ (or $N^i$) as a global section.
    \item $C$ contains all three kinds of points described in the previous cases and the flow $\varphi{|_C}$ is not parallelizable. There are points in $C$ that are not homoclinic but are limit of homoclinic points. These points are called \emph{dissonant points}.
\end{itemize}

In this section we shall deal with what we call \emph{global} non-saddle sets $K$ in closed manifolds $M$, namely isolated non-saddle sets whose region of influence is the whole phase space $M$. Such a non-saddle set $K$ is in fact a (global) unstable attractor with no external explosions, a notion introduced in \cite{AthPac} and studied in \cite{MSGS}, \cite{SGTrans} (their precise definition is recalled in the last section). The existing literature about these is mostly confined to the case when $K$ is connected, essentially because a component of an unstable attractor need not be an unstable attractor any more. Removing the connectedness assumption produces an important change of scenario: for example, there are no connected, global unstable attractors of this kind in $\mathbb{S}^n$ ($n \geq 1$) but it is straightforward to construct non-connected ones. By contrast, an isolated invariant set is non-saddle if and only if the same is true of its components, and so connectedness is largely irrelevant in this framework.

When $K$ is a global non-saddle set in a closed manifold all the points of $M\setminus K$ are homoclinic, and therefore $\varphi|_{M\setminus K}$ is parallelizable. We shall denote by $S$ any section of $M \setminus K$; precisely, $S$ is a closed subset of $M \setminus K$ such that the map $S \times \mathbb{R} \longrightarrow M \setminus K$ given by $(x,t) \longmapsto xt$ is a homeomorphism. For example, if $N = N^+ \cup N^-$ is an isolating block for $K$ then $S = N^o$ is such a global section. Exploiting this product structure of $M \setminus K$ we obtain the following result that characterizes the homotopy type of the quotient space $M/K$ and leads to relationships between the cohomologies of $M$, $K$ and the section $S$.

\begin{theorem}\label{thm:exact}
 Let $K$ be a global non-saddle set and $S$ a section of $M \setminus K$. Then, the quotient $M/K$ is homotopy equivalent to $\Sigma S\vee \mathbb{S}^1$, were $\Sigma S$ denotes the suspension of $S$. As a consequence the cup product in $\check{H}^*(M,K)$ is trivial. Moreover, there exists an exact sequence of the form
 \[
\cdots\longrightarrow\check{H}^{k-1}(S)\longrightarrow \check{H}^k(M)\longrightarrow\check{H}^k(K)\longrightarrow\check{H}^k(S)\longrightarrow \check{H}^{k+1}(M)\longrightarrow\cdots
 \]
\end{theorem}

\begin{proof} The quotient $M/K$ is just the one point compactification of $S\times\mathbb{R}$, which is homeomorphic to the quotient $S\times[-1,1]/S\times\{-1,1\}$. This quotient has the homotopy type of $\Sigma S\vee \mathbb{S}^1$ where $\Sigma$ denotes the suspension. In particular \[
\check{H}^*(M,K)\cong \check{H}^*(M/K,[K])\cong\check{H}^*(\Sigma S\vee \mathbb{S}^1,\{*\}),
\] where the first isomorphism comes from the strong excision property of \v Cech cohomology. Moreover, we have that there is a ring isomorphism
\[
\check{H}^*(\Sigma S\vee \mathbb{S}^1,\{*\})\cong \check{H}^*(\Sigma S,\{*\})\oplus \check{H}^*(\mathbb{S}^1,\{*\}).
\]

We prove that the ring structure in $\check{H}^*(M,K)$ is trivial by analyzing the expression just obtained. Since the product of two cohomology classes in $\check{H}^*(\mathbb{S}^1,\{*\})$ is trivial, it remains to see that so is the product of two classes in $\check{H}^*(\Sigma S,\{*\})$. This is a standard fact in cohomology theory but we include it here for the sake of completeness. 

Since $\Sigma S$ is connected we may choose the basepoint $*$ freely, and we take it to belong to $S$. The suspension $\Sigma S$ can be covered by two contractible open subsets $U_+$ and $U_-$, namely the cones $S\times (-1/2,1]/S\times\{1\}$ and $S\times[-1,1/2) /S\times\{-1\}$. Notice that the relative cup product satisfies 
\[
\smile:\check{H}^*(\Sigma S,U_+)\times \check{H}^*(\Sigma S,U_-)\longrightarrow\check{H}^*(\Sigma S,U_+\cup U_-)=\check{H}^*(\Sigma S,\Sigma S),
\]
and hence it must be zero. In addition, since $U_\pm$ are contractible, the inclusions $i_{\pm}:(\Sigma S, \{*\})\hookrightarrow (\Sigma S,U_{\pm})$ induce isomorphisms in the corresponding cohomology rings. It follows from the naturality of the cup product that, if $i:(\Sigma S, \{*\})\hookrightarrow (\Sigma S,\Sigma S)$ denotes the inclusion then, given cohomology classes $\alpha_1,\alpha_2\in \check{H}^*(\Sigma S,\{*\})$, we have that 
\[
\alpha_1\smile\alpha_2=i^*(i_+^{*-1}(\alpha_1)\smile i_-^{*-1}(\alpha_2))
\]
and as a consequence this product must be trivial.

The Mayer-Vietoris sequence for $\Sigma S = U_+ \cup U_-$ shows that $\check{H}^k(\Sigma S) \cong \check{H}^{k-1}(S)$ for $k \geq 2$ and $\check{H}^1(\Sigma S$ is the reduced \v{C}ech cohomology of $S$. Together with the preceding discussion, it follows that $\check{H}^k(M,K)\cong \check{H}^k(\Sigma S)\cong \check{H}^{k-1}(S)$ for $k \geq 2$ and $\check{H}^1(M,K)\cong\check{H}^1(\Sigma S)\oplus\mathbb{Z}\cong \check{H}^o(S)$. In addition, $\check{H}^0(M,K)\cong 0$. Replacing this in the long exact sequence of the pair $(M,K)$ yields the long exact sequence in the statement of the theorem.
\end{proof}

\begin{corollary}\label{cor:1}
    Let $K$ be a global non-saddle set and suppose that $H^1(M)=0$. Let $c$ be the number of connected components of $M \setminus K$. Then $K$ has $c+1$ connected components, of which (at least) one is an attractor and another one is a repeller. In particular, if $K$ has exactly $2$ components they form an attractor-repeller decomposition of $M$.
\end{corollary}

\begin{proof} Let $N$ be an isolating block of $K$ of the form $N=N^+\cup N^-$. As mentioned earlier, $N^o$ is a section of $M \setminus K$ so it also has $c$ components. If $\check{H}^1(M)=0$ it follows from Theorem~\ref{thm:exact} that $\check{H}^0(K)\cong \check{H}^0(N^o)\oplus \mathbb{Z}$, whence $K$ has $c+1$ components. Observe that $K$ is not an attractor  (otherwise no homoclinic components would be possible and $K = M$, but we always assume $K \subsetneq M$) and so $N^o$ must be nonempty. It follows that $K$ has at least two connected components.

Since $K$ and $N$ have the same number of connected components and from $\check{H}^0(K) \cong \check{H}^o(N^o) \oplus \mathbb{Z}$ we see that $N^o$ has one component less than $K$, it follows that at least one component of $N$ is disjoint from $N^o$. That component of $N$ is therefore positively invariant and the component of $K$ isolated by it is an attractor. The same argument with $N^i$ instead of $N^o$ proves that there is another component of $K$ which is a repeller.
\end{proof}

The fact that the cup product in $\check{H}^*(M,K)$ is trivial becomes especially useful when the cup product in $H^*(M)$ is nondegenerate, because then the rank of the image of $\check{H}^*(M,K)$ in $H^*(M)$ is bounded above by half the rank of $H^*(M)$. The following example illustrates the simplest case of this sort of analysis:

\begin{example} Let $K$ be a global non-saddle set in $\mathbb{RP}^n$, the real projective space of dimension $n$. We claim that $\check{H}^k(K;\mathbb{Z}_2) \neq 0$ for every $1 \leq k \leq n/2$.

To prove this consider the exact sequence \[\cdots \longrightarrow \check{H}^k(\mathbb{RP}^n,K) \longrightarrow H^k(\mathbb{RP}^n) \longrightarrow \check{H}^k(K) \longrightarrow \cdots\] where we omit the $\mathbb{Z}_2$ coefficients from the notation. Recall that $H^*(\mathbb{RP}^n)$ is the truncated polynomial ring $\mathbb{Z}_2[\alpha]/(\alpha^{n+1})$, where $\alpha$ is a generator of $H^1(\mathbb{RP}^n)$. Assume the generator $\alpha^k$ of $H^k(\mathbb{RP}^2)$ is the image of some element $z \in \check{H}^k(\mathbb{RP}^2,K)$ under the inclusion induced homomorphism. Then $\alpha^{2k} = 0$ because it is the image of $z \cup z$, which is zero by the triviality of the cup product in $(\mathbb{RP}^2,K)$. However, $\alpha^{2k}$ is the generator of $H^{2k}(\mathbb{RP}^n)$ and hence nonzero since $2k \leq n$. Thus in fact the arrow $H^k(\mathbb{RP}^2,K) \longrightarrow H^k(\mathbb{RP}^2)$ is zero, which entails ${\rm rk}\ \check{H}^k(K) \geq 1$ for each $1 \leq k \leq n/2$.
\end{example}

The following result relates the Euler characteristic of $K$ with those of $M$ and $N^o$ and has very strong consequences in dimension two.

\begin{proposition}\label{prop:Euler} Let $K$ be a global non-saddle set in a closed, connected manifold $M$. Let $S$ be a section of $M \setminus K$.  Then,   
    \[
    \chi(K)=\begin{cases}
    \chi(M) & \mbox{if}\; n\;\mbox{is even}\\
    \chi(S) & \mbox{if}\; n\;\mbox{is odd},
    \end{cases}
    \]
    where $n$ is the dimension of the manifold $M$.
\end{proposition}

\begin{proof}

    Through this proof we consider homology and cohomology with coefficients on $\mathbb{Z}_2$ in order to avoid considerations about orientability. (Recall that the Euler characteristic is independent of the choice of coefficients).
    
    Using the long exact sequence from Theorem~\ref{thm:exact} (which works also for field coefficients) it follows that 
    \[
    -\chi(S)=\chi(M)-\chi(K).
    \]
    Since $M$ is a closed manifold, Poincar\'e duality ensures that $\chi(M)=0$ whenever $n$ is odd and hence, we obtain the desired equality in this case.

    Suppose now that $n$ is even and set $\hat{M} = M \setminus K$. By Poincaré duality in the (open) $n$-manifold $\hat{M}$ we have $H_k(\hat{M}) \cong H^{n-k}_c(\hat{M})$, where $H^*_c$ denotes cohomology with compact supports. This cohomology is in turn isomorphic to $\check{H}^{n-k}(\hat{M}_{\infty},\infty)$, where $\hat{M}_{\infty}$ denotes the one-point compactification of $\hat{M}$. The reason is that a cohomology class in $\hat{M}$ has a compact support if and only if it extends to a cohomology class in $\hat{M}_{\infty}$ that vanishes on a neighbourhood of $\infty$ (a formal proof can be found in \cite[Corollary 12, p. 322]{Span}). Now, clearly $(\hat{M}_{\infty},\infty) \cong (M/K,[K])$ and by the computations of Theorem \ref{thm:exact} we have $H^{n-k}_c(\hat{M}) \cong \check{H}^{n-k}(\hat{M}_{\infty},\infty) \cong \check{H}^{n-k-1}(S)$. Finally, returning to Poincaré duality and bearing in mind that $\hat{M} \cong S \times \mathbb{R}$ we end up with \[H_k(S) \cong H_k(\hat{M}) \cong \check{H}^{n-k-1}(S) \cong H^{n-k-1}(S)\] where the last isomorphism owes to the fact that $S$ is an ANR because it is a retract of some open neighbourhood of itself. It then follows that $\chi(S) = 0$.
\end{proof}

Suppose that the phase space $M$ is a closed, connected $2$-manifold and $K \subset M$ is a global non-saddle set. Heuristically, we are mostly interested in the case when $K$ has the least possible number of components. Since $\chi(K) = \chi(M)$ is independent of $K$, such a $K$ will also have the smallest possible first Betti number, and so it will be the ``simplest'' possible from a homological point of view. The following results show that the (co)homology of $K$, together with orientability in some cases, characterizes $M$ completely.

\begin{theorem}
    Let $K$ be a global non-saddle set in a closed (and connected) $2$-manifold $M$. Then,
    \begin{enumerate}
        \item[(i)] If $M$ is the sphere $\mathbb{S}^2$ or the projective plane $\mathbb{RP}^2$, $K$ has at least two connected components. Moreover, if $K$ has exactly two connected components they form an attractor-repeller decomposition.

        \item[(ii)] If $K$ is connected and $M$ is orientable, then it is the connected sum of $\frac{1+\rank\check{H}^1(K)}{2}$ tori. (In particular ${\rm rk}\ \check{H}^1(K)$ is odd).

        \item [(iii)] If $K$ is connected and $M$ is nonorientable, then it is the connected sum of $1+\rank\check{H}^1(K)$ projective planes.
    \end{enumerate}
\end{theorem}

\begin{proof}
Case (i) follows directly from Corollary~\ref{cor:1} taking coefficients in $\mathbb{Q}$. %if $M=\mathbb{S}^2$. To prove it for $\mathbb{RP}^2$ notice that Proposition~\ref{prop:Euler} ensures that $\chi(K)=1$. If we assume that $K$ is connected it follows that $\rank\check{H}^1(K)=0$ and, hence, by \cite[Theorem~12]{BNLA} it has an isolating block $N=N^+\cup N^-$ that is a topological disc. Since $\partial N$ is connected and is contained either in $N^+$ or in $N^-$ it follows that $N$ is either positively or negatively invariant and, as a consequence, $K$ is either an attractor or a repeller. But since $K$ is compact this forces $K = M$, which we have discarded from the beginning.

%Suppose now that $K$ has two connected components. Then, since  $\chi(K)=1$, it follows that $\rank\check{H}^1(K)=1$ and so one of its components, say $K_1$, satisfies that $\rank\check{H}^1(K)=0$. Notice that $K_1$ must be non-saddle, and reasoning as before it follows that $K_1$ must be either an attractor or a repeller, being the other component its dual repeller or attractor.

To prove (ii) and (iii) it is sufficient to observe that since Corollary~\ref{cor:1} ensures that $\chi(M)=\chi(K)$ and $K$ is connected, it follows that $\chi(M)=1-\rank\check{H}^1(K)$. If $M$ is orientable of genus $g$ then $\chi(M) = 2 - 2g = 1 - {\rm rk}\ \check{H}^1(K)$ and so $g = (1+ {\rm rk}\ \check{H}^1(K))/2$; in particular, ${\rm rk}\ \check{H}^1(K)$ must be odd. If $M$ is nonorientable of genus $k$ then $\chi(M) = 2 - k = 1 - {\rm rk}\ \check{H}^1(K)$ and so $k = 1 + {\rm rk}\ \check{H}^1(K)$.
\end{proof}

\begin{corollary}[Dynamical characterization of closed surfaces] \label{cor:classification}
Let $M$ be a closed and connected $2$-manifold. We have the following mutually exclusive possibilities. 
\begin{enumerate}
\item[(i)] $M$ is $\mathbb{S}^2$ if and only if it admits a flow having a global non-saddle set that consists of two points.

\item[(ii)] $M$ is $\mathbb{RP}^2$ if and only if it admits a flow having a global non-saddle set that is the disjoint union of a circle and a point.

\item[(iii)] $M$ is the connected sum of $g\geq 1$ tori if and only if it is orientable and admits a flow having a global non-saddle set that is a wedge sum of $2g-1$ circles.

\item[(iv)] $M$ is the connected sum of $k \geq 2$ projective planes if and only if it is nonorientable and admits a flow having a global non-saddle set that is a wedge of $k-1$ circles.
\end{enumerate}

In each case the non-saddle set is the simplest possible: the one with the least number of connected components and the smallest first Betti number.
\end{corollary}

\begin{proof} The ``if'' part of (iii) and (iv) follows from the preceding theorem; for parts (i) and (ii) it is a direct consequence of the formula $\chi(M) = \chi(K)$. For the ``only if'' part we need to check that the corresponding manifolds admit the claimed flows. 

(i) If $M$ is $\mathbb{S}^2$ it is sufficient to consider the gravitational north-south flow that has the poles as fixed points and all the other points flow downwards along the meridians.

(ii) Suppose $M$ is $\mathbb{RP}^2$. Let $\mathbb{D}^2$ be the closed unit disk in $\mathbb{R}^2$ and denote its center by $p$. Consider an outwards radial flow in $\mathbb{D}^2$ that is stationary on $p$ and $\partial \mathbb{D}^2$, so that $(\{p\},\partial \mathbb{D}^2)$ is an attractor-repeller decomposition of $\mathbb{D}^2$. By identifying antipodal points in $\partial \mathbb{D}^2$ we obtain a flow in $\mathbb{RP}^2$ that has the pair $(\{p\},C)$ as an attractor-repeller decomposition, where $C$ is the image of $\partial \mathbb{D}^2$ in $\mathbb{RP}^2$ (which is a projective line; hence a topological circle). Thus $\{p\} \cup C$ is a global non-saddle set.

(iii) and (iv) We construct a flow on the torus or the Klein bottle that has a circle as a global non-saddle set. Consider a flow on the compact annulus $\mathbb{S}^1\times[0,1]$ that is stationary on the boundary and otherwise flows radially across the annulus so that $(\mathbb{S}^1\times\{0\},\mathbb{S}^1\times\{1\})$ is an attractor-repeller decomposition of the annulus. Identify $S^1\times\{0\}$ with $S^1\times\{1\}$ by an orientation preserving circle homeomorphism to obtain the desired flow on the torus and by an orientation reversing one to obtain the desired flow on the Klein bottle.

To finish the proof we shall construct a flow on a connected sum $M$ of $k>2$ projective planes having a wedge of $k-1$ circles as a global non-saddle set. (The case of connected sums of tori is similar but easier). Start from a connected sum of $k-2$ projective planes $M_0$ and consider the surface (with boundary) $S$ obtained from $M_0$ after removing the interiors of two disjoint open disks. The surface $M$ can be obtained from $S$ by attaching an annulus $A$ onto the two boundary components of $\partial S$: indeed, this operation is the same as making a connected sum a torus, which for non-orientable surfaces is equivalent to making a connected sum with two projective planes. We thus have a decomposition $M = S \cup A$, where $S$ is a regular neighbourhood of a wedge of $k-1$ circles that we denote by $K$. Observe that $S \setminus K$ is homeomorphic to $\partial S \times [0,+\infty)$, which is a disjoint union of two half-open annuli because $\partial S$ has two connected components. Each of these components is then pasted to a boundary component of $A$, and so $M \setminus K = A \cup (S \setminus K)$ is an open annulus. The desired flow is then constructed by letting all points in $K$ be stationary and carrying every other point across the annulus from one missing boundary component to the other (the flow needs to be appropriately slowed down as it approaches the missing boundary).
\end{proof}

\section{Non-saddle decompositions} \label{sec:nonsaddle2}

It is well known that if $M$ is compact and $A\subset M$ is an attractor with basin of attraction $\mathcal{A}(A)$, then the complement $R=M\setminus\mathcal{A}$(A) is a repeller and the dynamics on $M$ admits an easy description since the orbit of every point outside $A\cup R$ must be a connecting orbit between the two. The pair $(A,R)$ is a so-called \emph{attractor-repeller decomposition} of $M$ and $R$ is said to be the \emph{dual repeller} of $A$.   

Assume that $K$ is a non-global isolated non-saddle set and $M$ is compact. Since its region of influence $\mathcal{I}(K)$ of $K$ is a proper and invariant open subset of $M$, it follows that $L=M\setminus\mathcal{I}(K)$ is a non-empty compact invariant set. It is easy to see that $L$ is isolated. A natural question would be to ask whether this situation resembles that of the attractor-repeller decompositions just described, that is, if $L$ is also non-saddle and the orbit of every point not in $K\cup L$ is a connecting orbit between the two. This is not the case in general as the flow depicted in Figure~\ref{fig:dis} shows. It is a flow on the $2$-torus, with $K$ a meridian of the torus. In this case, neither $L=\{p\}$ is non-saddle nor every point outside $K\cup L$ lies in a connecting orbit between $K$ and $L$.

\begin{figure}[h]
    \centering
    \includegraphics[width=0.5\linewidth]{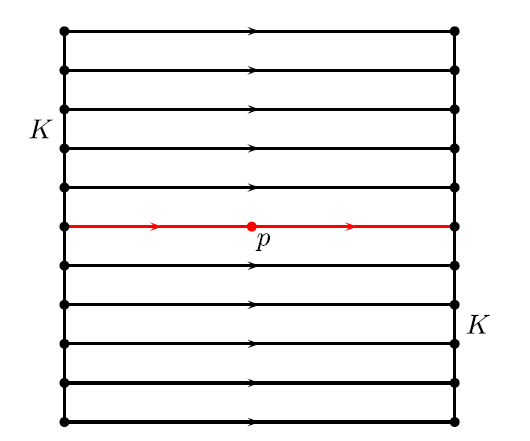}
    \caption{Flow on the torus.}
    \label{fig:dis}
\end{figure}

We shall say that a pair of isolated non-saddle sets $(K,L)$ is a \emph{non-saddle decomposition} of the compact metric space $M$ when the orbit of every point not in $K\cup L$ is a connecting orbit between $K$ and $L$. In this situation we say that $L$ is the \emph{dual non-saddle set} of $K$. The notion of non-saddle decompositions was introduced in \cite{GMRSo}. The following result shows necessary and sufficient conditions for the existence of the dual non-saddle set.

\begin{proposition}[Existence of non-saddle decompositions]\label{prop:dual}
    Suppose that $K$ is an isolated non-saddle set and $M$ is compact. The following are equivalent:
    \begin{enumerate}
        \item For every $x\in M\setminus K$, either $\omega(x)\cap K=\emptyset$ or $\omega^*(x)\cap K=\emptyset$; i.e. there are no orbits homoclinic to $K$.
        \item There exists an isolated non-saddle set $L\neq\emptyset$ such that the pair $(K,L)$ is a non-saddle decomposition.
    \end{enumerate}
\end{proposition}

\begin{proof}
     (1) $\Rightarrow$ (2) Let $L=M\setminus\mathcal{I}(K)$. Then, $L$ is an isolated invariant set and it is non-empty since it must contain either the $\omega$ or the $\omega^*$-limit of every point $x\in M\setminus K$. We see that $L$ is non-saddle. Let $N$ be an isolating block of $K$ of the form $N=N^+\cup N^-$ and consider $N_L=M\setminus\mathring{N}$. It is clear that $N_L$ is an isolating block of $L$ satisfying $N_L^o=N^i$ and $N_L^i=N^o$. To see that $L$ is non-saddle it is sufficient to see that $N_L$ is of the form $N_L=N_L^+\cup N_L^-$. Let $x\in N_L\setminus L=N_L\cap\mathcal{I}(K)$ and assume that $\omega^*(x)\subset K$. Then $x\in N_L^+$ since otherwise there would be some $t\geq 0$ such that $xt\in N_L^o\subset N^+$, and hence $\omega(x)\subset K$. Since both $\omega^*(x)$ and $\omega(x)$ are nonempty, this contradicts (1). In an analogous way if $\omega(x)\subset K$ it follows that $x\in N_L^-$ and, hence, $L$ is non-saddle. The fact that the orbit of every point not in $K\cup L$ is a connecting orbit is straightforward. 

    (2) $\Rightarrow$ (1) Conversely, if $(K,L)$ is a non-saddle decomposition every point in $M\setminus K$ is either in $L$, and then its limit sets must be contained in $L$, or it belongs to a connecting orbit and hence one of its limit sets is contained in $K$ and the other in $L$. 
\end{proof}

If $(K,L)$ is a non-saddle decomposition of a manifold it is possible to establish homological relations between the non-saddle sets and the phase space as the following result shows. 

\begin{theorem}\label{thm:decomp}
    Let $(K,L)$ be a non-saddle decomposition of a closed, connected and orientable $n$-dimensional manifold $M$. Then there is an exact sequence
    \[
\cdots\longrightarrow \check{H}_k(L)\longrightarrow H_k(M)\longrightarrow \check{H}^{n-k}(K)\longrightarrow \check{H}_{k-1}(L)\longrightarrow\cdots
    \]
\end{theorem}

\begin{proof}
    Consider the long exact sequence of singular homology of the pair $(M,M\setminus K)$
    \[
        \cdots\longrightarrow H_k(M\setminus K)\longrightarrow H_k(M)\longrightarrow H_k(M,M\setminus K)\longrightarrow H_{k-1}(M\setminus K)\longrightarrow\cdots
    \]

    By Alexander duality  we have that $H_k(M,M\setminus K)\cong\check{H}^{n-k}(K)$ so we only have to prove that $H_*(M\setminus K)\cong \check{H}_*(L)$. In order to see this notice that, as we have seen in the proof of Proposition~\ref{prop:dual}, if $N$ is an isolating block of $K$ of the form $N^+\cup N^-$ then $N_L=M\setminus\mathring{N}$ is an isolating block of $L$ of the form $N_L^+\cup N_L^-$. Hence $H_*(N_L)\cong \check{H}_*(L)$ and it suffices to prove that there is a strong deformation retraction of $M \setminus K$ onto $N_L$. Observe that $K\cup L$ is a global non-saddle set for $M$ so $M\setminus (K\cup L)$ is parallelizable and, since every trajectory intersects the local section $\partial N_L$, the latter must actually be a global section for the flow on $M \setminus (K \cup L)$. That is, for every $x\in M\setminus (K\cup L)$ there exists a unique $t_x$ such that $xt_x\in\partial N_L$; furthermore $t_x$ depends continuously on $x$. Consider the map $H:(M\setminus K)\times [0,1]\longrightarrow M\setminus K$ given by $H(x,s)=x(s\cdot t_x)$ if $x\in M\setminus (K\cup N_L)$ and $H(x,s)=x$ otherwise. This is a homotopy relative to $N_L$ between the identity and the map $r:M\setminus K\longrightarrow N_L$ given by $r(x)=xt_x$ if $x\in M\setminus (K\cup N_L)$ and $r(x)=x$ otherwise. Thus $r$ provides a strong deformation retraction between $M\setminus K$ and $N_L$ and so $H_*(M\setminus K)\cong H_*(N_L)\cong \check{H}_*(L)$ as required.
\end{proof}

\begin{corollary}
    Under the assumptions of Theorem~\ref{thm:decomp}, if $H_k(M)=H_{k-1}(M)=\{0\}$ then $\check{H}^{n-k}(K)\cong\check{H}_{k-1}(L)$ and, if $H_1(M)=\{0\}$, for instance, if $M=\mathbb{S}^n$ with $n\geq 2$, then $\check{H}_0(L)\cong\check{H}^{n-1}(K)\oplus\mathbb{Z}$.
\end{corollary}

\section{Continuations and bifurcations of non-saddle sets} \label{sec:bifurcations}

In this section we study robustness and bifurcation properties of isolated non-saddle sets. We start by recalling the basic notions of continuation theory (see \cite[Section~6]{Sal} for a detailed discussion). Let $\varphi_\lambda:M\times\mathbb{R}\longrightarrow M$ be a parametrized famility of flows depending continuously on a parameter $\lambda\in[0,1]$. The key observation for the theory of continuation is that being an isolating neighborhood is an open property. More precisely, suppose that $K_{\lambda_0}$ is an isolated invariant set for $\varphi_{\lambda_0}$ and $N_{\lambda_0}$ an isolating neighborhood for $K_{\lambda_0}$. Then there exists $\delta>0$ such that $N_{\lambda_0}$ is an isolating neighborhood for $\varphi_\lambda$ with $\lambda\in (\lambda_0-\delta,\lambda_0+\delta)\cap[0,1]$. For any such $\lambda$ we define $K_{\lambda}$ to be the maximal invariant subset in $N_{\lambda_0}$ for the flow $\varphi_{\lambda}$ and say that the family $(K_\lambda)$ of isolated invariant sets \emph{continues} or is a \emph{local continuation} of $K_{\lambda_0}$. An important observation is that if $N_1$ and $N_2$ are two isolating neighborhoods of $K_{\lambda_0}$ for $\varphi_{\lambda_0}$ then, the local continuations determined by these isolating neighborhoods coincide for $\lambda$ sufficienlty close to $\lambda_0$.

An interesting problem in continuation theory is to study under what circumstances a dynamical or topological feature of an isolated invariant set is preserved by local continuations. For instance, if $K_{\lambda_0}$ is an attractor, it continues to a family $(K_\lambda)$ of attractors for $\lambda$ close to $\lambda_0$ with the additional property that the \v Cech homologies and cohomologies of the $K_\lambda$ coincide with that of $K_0$ \cite[Theorem~4]{Sanjuni}. However, the property of non-saddleness is not preserved by continuations as illustrated in Figure \ref{fig:cont}, where the non-saddle circumference $K_0$ breaks up into an arc for $\lambda >0$ and is no longer non-saddle. The homological properties of $K_0$ are not preserved either. Actually, the preservation of the dynamical property of non-saddleness is strongly related to the preservation of some of its homological properties, at least for smooth dynamics (see \cite{Bdcds, BSdis, GSRo}). 

\begin{figure}[h]
    \centering
    \includegraphics[width=0.5\linewidth]{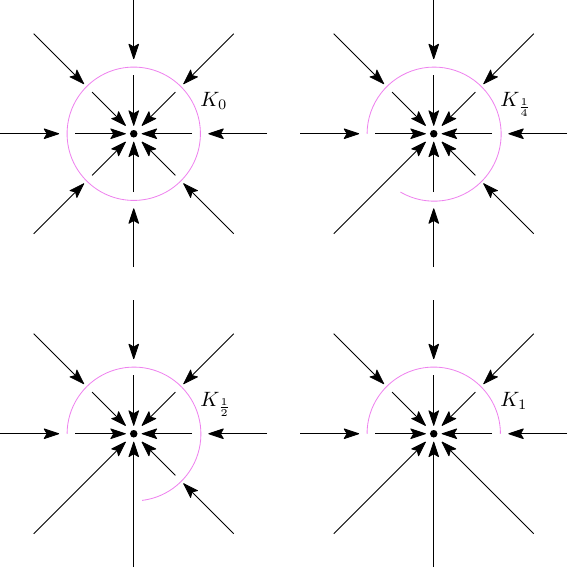}
    \caption{Continuations of an isolated non-saddle set may be saddle}
    \label{fig:cont}
\end{figure}

The following result describes a natural setup where a non-saddle set continues to non-saddle sets.

\begin{theorem} Suppose that $K$ is an invariant set for all $\lambda \in [0,1]$ and an isolated non-saddle set for $\lambda = 0$. Let $(K_{\lambda})$ be a continuation of $K$. Then there exists $\lambda_0 > 0$ such that $K_{\lambda}$ is an isolated non-saddle set for all $\lambda < \lambda_0$. Moreover, these $K_{\lambda}$ all satisfy $\check{H}^*(K_{\lambda}) \cong \check{H}^*(K)$.
\end{theorem}
\begin{proof} For $\lambda = 0$ fix an isolating block for $K$ with the form $N = N^+ \cup N^-$. The compact set $N^+$ is positively invariant for $\lambda = 0$. It may not be positively invariant for $\lambda > 0$, but the following holds: for any compact $P \subset N^+$ disjoint from $N^i$ there exists $\lambda_0$ such that $\varphi_{\lambda}(P \times [0,+\infty)) \subset N^+$ for every $\lambda < \lambda_0$. If this were not the case there would exist sequences $(x_n) \subset P$ and $\lambda_n \rightarrow 0$ such that the forward orbit of $x_n$ under $\varphi_{\lambda_n}$ is not entirely contained in $N^+$. Since $K$ is invariant for $\varphi_{\lambda_n}$, the forward orbit of $x_n$ cannot intersect it; moreover, $N^+ \setminus K$ and $N^- \setminus K$ are clopen subsets of $N \setminus K$ and so (by connectedness) the forward orbit of $x_n$ cannot go from $N^+ \setminus K$ to $N^- \setminus K$ without exiting $N$ first. Thus for the orbit of $x_n$ to escape $N^+$ it must do so through $N^+ \cap \partial N = N^i$. Let $t_n \geq 0$ be the first time the orbit of $x_n$ hits $N^i$, and also set $y_n$ to be the point at which this happens; that is, $y_n := \varphi_{\lambda_n}(x_n,t_n) \in N^i$. Passing to subsequences we may assume $x_n \rightarrow x \in P$ and $y_n \rightarrow y \in N^i$. If $t_n$ had a subsequence convergent to $0$ we would have $y = x$ but this is not possible because $x \in P$ which is disjoint from $N^i$. Hence $t_n \geq \epsilon > 0$ for some small $\epsilon$ and every $n$. Now for any $0 \leq t \leq \epsilon$ one would have $\varphi_0(y,-t) = \lim_n \varphi_{\lambda_n}(x_n,t_n-t) \in N^+$ because $t_n - t \geq 0$. This is not possible since $y$ is a transverse entry point into $N$ for $\lambda = 0$, and so $\varphi_0(y,(-\delta,0)) \cap N = \emptyset$ for small enough $\delta>0$.

Pick another isolating block $N_1$ for $K$, again with the form $N_1 = N_1^+ \cup N_1^-$, and contained in the interior of $N$. Let $K_{\lambda}$ be the continuation of $K$ determined by $N_1$. We are going to show that the $K_{\lambda}$ are non-saddle sets by finding isolating neighbourhoods $B_{\lambda}$ of the form $B_{\lambda}^+ \cup B_{\lambda}^-$ for them as follows. By the discussion above applied to $P = N_1^+$ there exists $\lambda_0 > 0$ such that $\varphi_{\lambda}(N_1^+ \times [0,+\infty)) \subset N^+$ and (since the argument works the same for compact subsets of $N^-$) also $\varphi_{\lambda}(N_1^- \times (-\infty,0]) \subset N^-$ for $\lambda < \lambda_0$. Set $B_{\lambda}$ to be the closure of $\varphi_{\lambda}(N_1^+ \times [0,+\infty)) \cup \varphi_{\lambda}(N_1^- \times (-\infty,0])$. This contains $N_1^+ \cup N_1^- = N_1$, and is contained in $N$. By construction $\varphi_{\lambda}(N_1^+ \times [0,+\infty))$ and $\varphi_{\lambda}(N_1^- \times [0,+\infty))$ are, respectively, positively and negatively invariant under $\varphi_{\lambda}$ and so $B_{\lambda} = B_{\lambda}^+ \cup B_{\lambda}^-$. Finally, for small enough $\lambda$ the maximal invariant sets of $N$ and $N_1$ coincide (namely, they are the $K_{\lambda}$) and so the same is true for $B_{\lambda}$. Thus the $K_{\lambda}$ are non-saddle sets.

To check that $\check{H}^*(K_{\lambda}) = \check{H}^*(K)$ consider the following chain of homomorphisms induced by the chain of inclusions $K\subset K_\lambda\subset N_1\subset B_\lambda$:
\[
\check{H}^*(B_{\lambda})\longrightarrow \check{H}^*(N_1)\longrightarrow \check H^*(K_\lambda)\longrightarrow \check{H}^*(K).
\]
Both inclusions $K_{\lambda} \subset B_{\lambda}$ and $K \subset N_1$ induce isomorphisms because of the form of $B_{\lambda}$ and $N_1$, and so the compositions of any two consecutive arrows is an isomorphism. It follows that all the arrows are isomorphisms and so $\check{H}^*(K) \cong\check{H}^*(K_\lambda)$ as we wanted to see. 
\end{proof}

We now analyze the setup of the preceding theorem from the perspective of bifurcation theory. Since $K$ is invariant for all $\lambda \in [0,1]$, the $K_{\lambda}$ contain $K$; i.e. during the continuation the set $K$ may grow larger and ``spill over'' into a component $U$ of $M \setminus K$. Assuming that $K$ retains its non-saddle nature along the whole continuation, the next theorem states that a family of invariant sets $S_{\lambda}$ is expelled by $K$ into $U$ and characterizes when these are non-saddle. When $U$ has several ends (i.e. connected components at infinity) corresponding to $K$ the set $S_{\lambda}$ may have several components, each arising from a different end. Moreover, since $K$ is non-saddle, ``as seen'' from an end of $U$ it must behave as an attractor or a repeller, but it may behave differently at different ends of $U$. To avoid the notational burden of having to distinguish ends of $U$ and components of $S_{\lambda}$ we shall require that $K$ be connected and $M$ be a closed, connected manifold with $H^1(M) = 0$. As shown in \cite[Theorem~25]{BSdis} this guarantees that: (i) if $N$ is an isolating block of $K$ of the form $N=N^+\cup N^-$, different components of $N\setminus K$ lie in different components of $M\setminus K$; (ii) for every component $U$ among the finitely many components of $M \setminus K$, the set $K$ is an attractor or a repeller for the restricted flow $\varphi|_{K \cup U}$. Intuitively, (i) implies that every component $U$ of $M \setminus K$ has only one end corresponding to $K$ and (ii) formalizes the idea that $K$ behaves as an attractor or a repeller as seen from the component $U$.

\begin{theorem}\label{thm:bifur} Assume that $M$ is a closed and connected manifold with $H^1(M) = 0$. Suppose that $K$ is connected and isolated non-saddle for each $\lambda\in [0,1]$. Let $U$ be a component of $M \setminus K$ such that $K_{\lambda} \cap U \neq \emptyset$ for every sufficiently small $\lambda$. Then:
    \begin{itemize}
        \item[(1)] There exists a family of isolated invariant sets $(S_\lambda)$ for $\varphi_\lambda$, each contained in $U$, that converge upper semicontinuosly to $K$ as $\lambda\to 0$.
        \item[(2)] Moreover, $S_\lambda$ is non-saddle if and only if $S_\lambda$ separates $U$. In particular, $S_\lambda$ is an attractor (repeller) whenever $K$ changes its stability with respect to $U$ from attracting (repelling) to repelling (attracting).
    \end{itemize}
\end{theorem}

%\begin{theorem}\label{thm:bifur}
 %   Suppose that $K$ \textcolor{red}{Hay que pedir conexión} is isolated non-saddle for each $\lambda\in [0,1]$. Then, there exists $\lambda_0>0$ such that $K$ continues to a family of  isolated non-saddle sets $(K_\lambda)$ for $\lambda<\lambda_0$ with $\check{H}^*(K_\lambda)\cong\check{H}^*(K)$ satisfying that for each component $U$ of $M\setminus K$ with $K_\lambda\cap U\neq\emptyset$ for each $\lambda$ sufficiently small there exists a family of isolated invariant sets $(S_\lambda)$ for $\varphi_\lambda$, each contained in $U$, that converge upper semicontinuosly to $K$ as $\lambda\to 0$. Moreover, $S_\lambda$ is non-saddle if and only if $S_\lambda$ separates $U$. In particular, $S_\lambda$ is an attractor (repeller) whenever $K$ changes its stability with respect $U$ from attracting (repelling) to repelling (attracting).
%\end{theorem}

\begin{proof}

Since $K$ is non-saddle for every $\lambda$, we have that if we fix $\lambda$, $\partial U$ is either an attractor or a repeller for $\varphi_\lambda|_{\overline{U}}$. Suppose that $\partial U$ is a repeller for $\lambda=0$ and for the fixed $\lambda$, i.e. there is no change of stability. We call in this case $R_\lambda=K_\lambda\cap\overline{U}$. Since $\partial U$ is a repeller for this $\lambda$, it is in particular a repeller for the restriction flow $\varphi_\lambda|_{R_\lambda}$. Then, the claimed $S_\lambda$ is the dual attractor of $\partial U$ for $\varphi_\lambda|_{R_\lambda}$.

Let us see that $S_\lambda$ is non-saddle if and only if it separates $U$. Since $\overline{U}$ is compact, then there is an attractor $A'_\lambda$ for $\varphi_\lambda|_{\overline{U}}$ dual to $R_\lambda$. Notice that if we choose $\lambda$ sufficiently small, $A'_\lambda$ is connected being a continuation of the dual attractor of $\partial U$ for $\varphi_0|_{\overline{U}}$. 

Suppose that $S_\lambda$ is non-saddle and see that $W:=\mathcal{R}_\lambda(\partial U)\cap\mathcal{A}_\lambda(A'_\lambda)=\emptyset$. Since $U\setminus S_\lambda=(\mathcal{R}_\lambda(\partial U)\cup\mathcal{A}_\lambda(A'_\lambda))\cap U$ and both basins are open sets, the result follows. Let $P$ be an isolating block for $A'_{\lambda}$. Since $H^1(M)=0$ it follows that $A'_{\lambda}$ does not separate its basin of attraction $U\setminus R_\lambda$. Otherwise there would be a component of $M\setminus A'_\lambda$ whose closure is contained in $U\setminus R_\lambda$, contradicting that $A'_\lambda$ is the dual attractor of $R_\lambda$ for $\varphi_\lambda|_{\overline{U}}$. As a consequence $\partial P$ is connected being a section of the restriction of $\varphi_\lambda$ to the connected set $U\setminus (A'_\lambda\cup R_\lambda)$. Notice that every point in $\partial P$ either connects $A'_\lambda$ with $S_\lambda$ or with $\partial U$. It follows that,
\[
\partial P=(\partial P\cap W) \cup (\partial P\cap\mathcal{I}_\lambda(S_\lambda)).
\]
Since $S_\lambda$ is non-saddle, it follows that $\mathcal{I}_\lambda(S_\lambda)$ is open, and, being $W$ and $\mathcal{I}_\lambda(S_\lambda)$ disjoint  open sets it follows that $\partial P$ is either contained in $W$ or in $\mathcal{I}_\lambda(S_\lambda)$. Notice that the first possibility must be excluded, because otherwise $\mathcal{I}_\lambda(S_\lambda)$ must be contained in $R_\lambda$ and, $\partial \mathcal{I}_\lambda(S_\lambda)\subset R_\lambda$ would be a compact invariant set disjoint from $S_\lambda$ and, hence, would be contained in $\partial U$, thus, since $U$ is connected, $U=\mathcal{I}_\lambda(S_\lambda)$ which is not possible $A'_\lambda \subset U$ being non-empty and disjoint with $\mathcal{I}_\lambda(S_\lambda)$. Hence, $\partial P\subset\mathcal{I}_\lambda(S_\lambda)$ which ensures that $W=\emptyset$, since any trajectory connecting $\partial U$ and $A'_\lambda$ must intersect $\partial P$.

Conversely, suppose that $S_\lambda$ separates $U$. Since $\partial U$  is connected, one of the components of $U\setminus S_\lambda$ must have its closure contained in $U$. Thus, $A'_\lambda$ must lie in this component since, otherwise there would be a point not in $A'_\lambda\cup R_\lambda$, whose trajectory is not a connecting orbit between them, contradicting that they form an attractor-repeller decomposition. As a consequence, $\mathcal{R}_\lambda(\partial V)\cap\mathcal{A}_\lambda(A'_\lambda)=\emptyset$ and, hence, by \cite[Theorem~1]{BSbif} $S_\lambda$ is non-saddle.

Suppose that $K$ changes its stability form attracting to repelling with respect to $U$. In such a case, $K_\lambda$ is an attractor so is $S_\lambda$ being an attractor for the restriction of $\varphi_\lambda$ to an attractor.

The remaining cases are completely analogous. 
\end{proof}

The next corollary is formulated in terms of the Borsuk homotopy type (or ``shape type''). This was introduced by Borsuk (\cite{Bormono}) as a modification of the usual homotopy type and coincides with it on spaces having a nice local topology, such as manifolds. On compact spaces with bad local properties the Borsuk homotopy type is often more useful than the usual one because it overlooks the bad local structure; for instance, the Warsaw circle and $\mathbb{S}^1$ have the same Borsuk homotopy type. Details about this theory can be found in \cite{DySe}, \cite{MarSe}.

\begin{corollary}
    Under the assumptions of Theorem~\ref{thm:bifur}, assume $U$ is contractible. Then $S_\lambda$ is non-saddle if and only if $\check{H}^*(S_\lambda)\cong \check{H}^*(\mathbb{S}^{n-1})$. In particular, if $M=\mathbb{S}^2$ then $S_\lambda$ is non-saddle if and only if it has the Borsuk's homotopy type of $\mathbb{S}^1$. Otherwise it has the Borsuk homotopy type of a totally disconnected compactum. 
\end{corollary}

\begin{proof} Since $U$ is contractible it is orientable and we can apply Alexander duality with integer coefficients. Then,
\[
\check{H}^k(S_\lambda)\cong H_{n-k}(U,U\setminus S_\lambda)\cong\widetilde{H}_{n-k-1}(U\setminus S_\lambda)
\]
where the last isomorphism comes from the long exact sequence of the pair $(U,U\setminus S_\lambda)$.

Suppose that $S_\lambda$ is non-saddle and assume, without loss of generality, that $\partial U$ is a repeller for $\varphi_\lambda$. Hence, as we have seen in the proof of Theorem~\ref{thm:bifur}, $U\setminus S_\lambda$ is the disjoint union of $\mathcal{R}_\lambda(\partial U) \setminus \partial U$ and $\mathcal{A}_\lambda(A'_\lambda)$. Notice that $\mathcal{A}_\lambda(A'_\lambda)$ is contractible because it is the basin of attraction of an attractor that continues the dual attractor $A$ of $\partial U$ for $\varphi_0|_{\overline{U}}$. In addition, $\mathcal{R}_\lambda(\partial U) \setminus \partial U$ is the complement in $U$ the dual attractor $\hat{A}_\lambda$ of $\partial U$ for $\varphi_\lambda|_{\overline{U}}$. Since the basin of attraction of $\hat{A}_\lambda$ is $U$ it follows that $\check{H}^*(\hat{A}_\lambda)\cong H^*(U)\cong H^*(\{*\})$.

Suppose that $S_\lambda$ is non-saddle and assume, without loss of generality, that $\partial U$ is a repeller for $\varphi_\lambda$. Hence, as we have seen in the proof of Theorem~\ref{thm:bifur}, $U\setminus S_\lambda$ is the disjoint union of $\mathcal{R}_\lambda(\partial U)\setminus \partial U$ and $\mathcal{A}_\lambda(A'_\lambda)$. We compute the homology of these two sets as follows. Recall that $A'_{\lambda}$ is a continuation of $A$, the attractor dual to the repeller $\partial U$ for $\varphi_0|_{\overline{U}}$. Since the basin of attraction of $A$ is $U$, which is contractible, it has the cohomology of a point and the same is true of its continuation $A'_{\lambda}$. Therefore $\mathcal{A}_{\lambda}(A'_{\lambda})$ also has the (co)homology of a point. As for $\mathcal{R}_\lambda(\partial U)\setminus \partial U$, it is the complement in $U$ of the dual attractor $\hat{A}_\lambda$ of $\partial U$ for $\varphi_\lambda|_{\overline{U}}$. Since the basin of attraction of $\hat{A}_\lambda$ is $U$, once again $\hat{A}_{\lambda}$ has the cohomology of a point and then by Alexander duality \[\check{H}^k(\{*\}) \cong 
\check{H}^k(\hat{A}_\lambda)\cong \widetilde{H}_{n-k}(U,\mathcal{R}_\lambda(\partial U)\setminus \partial U)\cong\widetilde{H}_{n-k-1}(\mathcal{R}_\lambda(\partial U)\setminus \partial U).
\] so $\mathcal{R}_{\lambda}(\partial U)\setminus \partial U$ has the homology of $\mathbb{S}^{n-1}$.

It follows that $\mathcal{R}_{\lambda}(\partial U)\setminus \partial U$ is connected and so is $\mathcal{A}_{\lambda}(A'_{\lambda})$ because it has the homology of a point. Thus $U \setminus S_{\lambda} = (\mathcal{R}_{\lambda}(\partial U)\setminus \partial U) \cup \mathcal{A}_{\lambda}(A'_{\lambda}))$ has two connected components and 
\[
\check{H}^{n-1}(S_\lambda)\cong\widetilde{H}_0(U\setminus S_\lambda) \cong\mathbb{Z}.
\]
For $k\neq n-1$, 
\begin{multline*}
\check{H}^k(S_\lambda)\cong\widetilde{H}_{n-k-1}(U\setminus S_\lambda)\cong H_{n-k-1}(U\setminus S_\lambda)\\
\cong  H_{n-k-1}(\mathcal{R}_\lambda(\partial U)\setminus \partial U)\oplus H_{n-k-1}(\mathcal{A}_\lambda(A'_\lambda))\cong \check{H}^k(A'_\lambda).
\end{multline*}
Therefore $\check{H}^*(S_\lambda)\cong\check{H}^*(\mathbb{S}^{n-1})$, as claimed.

Conversely, if $S_\lambda$ satisfies that $\check{H}^*(S_\lambda)\cong \check{H}^*(\mathbb{S}^{n-1})$, Alexander duality ensures that
\[
\widetilde{H}_0(U\setminus S_\lambda)\cong\check{H}^{n-1}(S_\lambda)\cong\mathbb{Z}.
\]
Hence $S_\lambda$ separates $U$ and, by Theorem~\ref{thm:bifur}, it follows that $S_\lambda$ is non-saddle.

The second part of the statement follows from the fact that, if $M=\mathbb{S}^2$, then each component of $M\setminus K$ is contractible so $S_\lambda$ is non-saddle if and only if $\check{H}^*(S_\lambda)\cong\check{H}^*(\mathbb{S}^1)$. For subcompacta of the $2$-sphere this is equivalent to having Borsuk's homotopy type of $\mathbb{S}^1$. Notice that if $S_\lambda$ is saddle then $U\setminus S_\lambda$ is connected and, hence, so is $\mathbb{S}^2\setminus S_\lambda$. Therefore, each component of $S_\lambda$ must have the Borsuk's homotopy type of a point.
\end{proof}

\section{$W$-sets} \label{sec:Wsets1}

In this section we generalize the notion of a non-saddle set to encompass a bigger class of isolated invariant sets. We also obtain some relationships for this new class that generalize the classical Morse inequalities for flows defined on $2$-dimensional manifolds. 

We recall that an isolated non-saddle set $K$ is characterized by the fact that the trajectories of points nearby remain close to $K$ either in positive or negative time, ultimately being attracted or repelled, respectively, by $K$. To obtain our new class of compacta we relax this condition. We start by assuming the general situation in which $\varphi:M\times\mathbb{R}\longrightarrow M$ is a flow defined on an $n$-dimensional manifold.

\begin{definition}
    Let $K\subset M$ be a compact invariant set. We say that a point $x\in M$ is \emph{witnessed} by $K$ when its trajectory intersects every neighbourhood of $K$. We say that $K$ is a \emph{$W$-set} if it possesses a neighborhood $U$ all whose points are witnessed by $K$. The maximal $U$ for which this property is satisfied is called the \emph{witness region of} $K$ and is denoted by $\mathcal{W}(K)$. A $W$-set is said to be \emph{global} if its witness region is the whole phase space.
\end{definition}

It is easy to check that $x$ is witnessed by $K$ if and only if $(\omega(x) \cup \omega^*(x))\cap K\neq\emptyset$. It is clear that attractors, repellers and isolated non-saddle sets are $W$-sets. It is straightforward to see that the witness region of an isolated non-saddle set coincides with its region of influence and analogous relationships can be obtained for attractors and repellers. The following examples show that $W$-sets also encompass other well-known classes of isolated invariant compacta such as isolated unstable attractors and isolated weak attractors.

\begin{example}\label{ex:isoatr}
    Let $K$ be an unstable attractor. Recall that this is a compact invariant set that is possibly not stable but satifies that the set $\mathcal{A}(K)$ of points whose $\omega$-limit is non-empty and contained in $K$ is a neighborhood of $K$. Clearly every point $x \in U = \mathcal{A}(K)$ is witnessed by $K$, so $K$ is a $W$-set with $\mathcal{W}(K) \supseteq \mathcal{A}(K)$. In fact $\mathcal{W}(K)=\mathcal{A}(K)$: if a point $x$ is witnessed by $K$, its trajectory must visit any neighborhood of $K$ so in particular it must visit $\mathcal{A}(K)$. Since the latter is invariant, $x \in \mathcal{A}(K)$.

    For a specific example of an isolated unstable attractor see Figure~\ref{fig:mendelson}. The flow depicted in the picture is known as the \emph{Mendelson flow}. The point $q$ is a saddle fixed point whose basin of attraction $\mathcal{A}(\{q\})=\mathbb{R}^2\setminus\{p\}$ and, hence $\{q\}$ is an isolated unstable attractor. Notice that in this case, the isolated invariant compactum $K=\{p,q\}$ is a global $W$-set. 
\end{example}

\begin{figure}[h]
    \centering
    \includegraphics[width=0.7\linewidth]{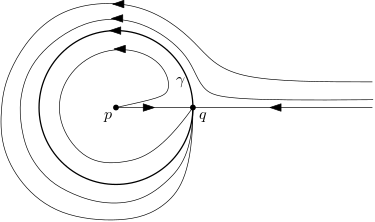}
    \caption{Mendelson flow}
    \label{fig:mendelson}
\end{figure}

\begin{remark}
    Reasoning as in Example~\ref{ex:isoatr} it is easy to see that in the case of attractors, repellers and isolated non-saddle sets, their witness regions coincide respectively with their basins of attraction, repulsion and region of influence.
\end{remark}

\begin{example} The following definitions comes from \cite[Chapter~V]{BhSz}. Suppose that $K$ is a compact invariant set and define $\mathcal{A}_w(K)$ as the set of points of $M$ whose $\omega$-limit has a non-empty intersection with $K$. The set $K$ is said to be a \emph{weak attractor} whenever $\mathcal{A}_w(K)$ is a neighborhood of $K$. As in the case of unstable attractors, it is clear that weak attractors are $W$-sets with $\mathcal{W}(K)=\mathcal{A}_w(K)$.

    For an example of a weak attractor that is not an unstable attractor see Figure~\ref{fig:mendmod}. This picture shows a flow that is a slight modification of the Mendelson flow. In this case, the fixed point $q$ (which is saddle as an invariant set) is a weak attractor with $\mathcal{A}_w(\{q\})=\mathbb{R}^2\setminus\{p\}$. As before, the  invariant compactum $K=\{p,q\}$ is a global $W$-set. Notice that, in this case, neither the $\omega$ nor the $\omega^*$-limit of any point different from $p$ and $q$ is contained in $K$. 
    
    In order to obtain an example of a $W$-set that is a saddle set and is neither a weak nor an unstable attractor we can modify this flow in such a way that the curve $\gamma$ locally repels the bounded component of its complement.
\end{example}

\begin{figure}[h]
    \centering
    \includegraphics[width=0.5\linewidth]{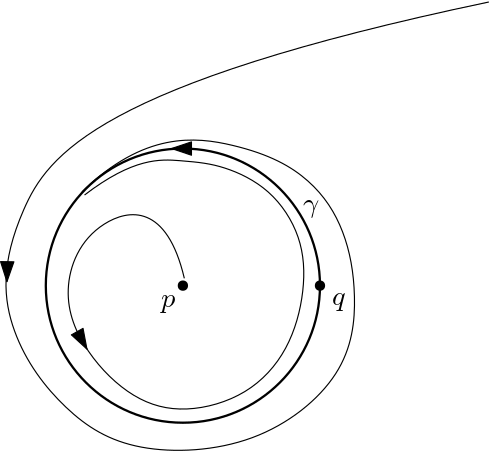}
    \caption{Modification of the Mendelson flow}
    \label{fig:mendmod}
\end{figure}

\begin{example}
    Suppose that $M$ is compact and $\{M_1,M_2,\ldots,M_k\}$ is a \emph{Morse decomposition} of $M$. That is, each $M_i$ is an isolated invariant set and each point in $M\setminus (\cup_{i=1}^k M_i)$ satisfies that there exist $i$ and $j$ with $i<j$ such that $\omega(x)\subset M_i$ and $\omega^*(x)\subset M_j$. 

    The isolated invariant set $K=\cup_{i=1}^k M_i$ is a global $W$-set. By the Conley-Morse equations of the decomposition one has $\chi h(K)=\chi(M)$.
\end{example}

The following result shows that the previous equality between the Euler characteristics holds in the more general case of global $W$-sets whenever the phase space is a closed $2$-manifold.

\begin{theorem} \label{teo:main} Let $K$ be an isolated, global $W$-set for a flow in a closed $2$-manifold $M$. Then $\chi(M) = \chi h(K)$.
\end{theorem}

\begin{proof} Let $N$ be an isolating block for $K$. As mentioned in Section \ref{sec:background} we can take $N$ to be a compact $2$--manifold and $N^i,N^o$ to be compact $1$--manifolds (with boundary) in $\partial N$.

Set $E := \overline{M \setminus N}$, which is a compact $2$--manifold with boundary $\partial E = \partial N$. Let $p_1,\ldots,p_t$ be the tangency points in $\partial N$ (if there are none, skip ahead to the definition of $L$). Notice that, viewed from $E$, these are interior tangencies. For each them there exists $\epsilon > 0$ such that $p_i \cdot (0,\epsilon) \subset M \setminus N = {\rm int}\ E$ and, since the forward orbit of $p_i$ cannot remain in $E$ forever, there must exit a time $t_i > 0$ such that $p_i \cdot (0,t_i) \subset {\rm int}\ E$ and $q_i := p_i \cdot t_i \in \partial E$.

Denote by $\gamma_i$ the closed arc $p_i \cdot [0,t_i]$. We observe that $\gamma_1 \cup \ldots \cup \gamma_t$ is positively invariant in $E$. To check this consider a point $p$ in some $\gamma_i$ and follow its forward orbit while it remains in $E$. The point will travel along $\gamma_i$ until it first hits $\partial N$ at $q_i$. This point must belong to $N^i$ because its past trajectory comes from the interior of $E$ (hence from the complement of $N$). If $q_i$ is a transverse entry point into $N$ the forward orbit segment of $p$ inside $E$ ends here. Otherwise $q_i$ coincides with one of the other tangency points $p_j$ and we repeat the argument following the arc $\gamma_j$.

Set $L := N^i \cup \gamma_1 \cup \ldots \cup \gamma_t \subset E$. If there are no tangency points in $\partial N$ we just set $L = N^i$. Evidently $L$ is a closed subset of $E$, and it is positively invariant in $E$: we already checked this for points in $\gamma_1 \cup \ldots \cup \gamma_t$ in the previous paragraph, and for any $p \in L$ outside the arcs its forward trajectory in $E$ consists just of $\{p\}$ itself because $p$ is a transverse entry point into $N$. One also checks easily that $L$ is an exit set for $E$, since the smaller set $N^i$ already is. Thus $(E,L)$ is an index pair. Every point in $E$ must exit $E$ in the future (respectively, past) since otherwise its $\omega$-limit set (respectively, $\omega^*$-limit) would be contained in $E$ and this contradicts the assumption that $K$ is a $W$-set. Thus, the maximal invariant subset of $E$ is empty and its Conley index is zero, and so $(E/L,*)$ has the homotopy type of $(*,*)$.

By the additivity of the Euler characteristic we have \[\chi(E,N^i) = \chi(E,L) + \chi(L,N^i) = \chi(E/L,[L]) + \chi(L/N^i,[N^i])\] The previous paragraph shows that $\chi(E/L,*) = \chi(*,*) = 0$. Also, since each arc $\gamma_i$ intersects $N^i$ exactly at its endpoints, collapsing $N^i$ inside $L$ is the same as collapsing the endpoints of the arcs $\gamma_i$ to a single point, and this yields $\vee^t \mathbb{S}^1$ because the interiors of the arcs are mutually disjoint. It thus follows from the expression above that \[\chi(E,N^i) = \chi(\vee^t \mathbb{S}^1,*) = -t.\]

Writing $(M,\partial N) = (E,N^i) \cup (N,N^o)$ we have \[\chi(M,\partial N) = \chi(E,N^i) + \chi(N,N^o) - \chi(\partial N,N^i \cap N^o).\] Since $\chi(\partial N) = 0$ and $N^i \cap N^o$ consists of the $t$ tangency points $p_1,\ldots,p_t$,  we get \[\chi(M) = -t + \chi h(K) - (-t) = \chi h(K).\]
\end{proof}

If $M$ has a Morse decomposition $\{M_1, M_2,\ldots,M_k\}$ it is well known that $M_1$ is an attractor and $M_k$ is a repeller. In particular, the union of the Morse sets has at least a component that is an attractor and another one that is a repeller. The following result shows that a similar situation occurs with $W$-sets in the $2$-sphere $\mathbb{S}^2$.

\begin{theorem}\label{teo:positive}
Let $K$ be an isolated global $W$-set in the $2$-sphere $\mathbb{S}^2$. Then, $K$ has at least two components such that each of them is either an attractor or a repeller that does not separate $\mathbb{S}^2$. Moreover, if $K$ has exactly two connected components, they must form an attractor-repeller decomposition.
\end{theorem}

To prove the theorem we first make a general remark concerning the Euler characteristic of the Conley index in surfaces. Suppose $N$ is a connected isolating block and denote by $K$ its maximal invariant subset. If $K$ is neither an attractor nor a repeller then both $N^o$ and $N^i$ must be nonempty, so $H_0(N,N^o) =  0$ and $H_2(N,N^o) = H^0(N,N^i) = 0$. Hence in this case $\chi h(K) = -{\rm rk}\ H_1(N,N^o) \leq 0$. If $K$ is an attractor or a repeller then $\chi h(K) = 1 - {\rm rk}\ \check{H}^1(K)$, so we still have $\chi h(K) \leq 0$ unless ${\rm rk}\ \check{H}^1(K) = 0$, in which case $\chi h(K) = 1$. We can exploit this to prove the following result:

\begin{proof} Let $N$ be an isolating block for $K$ and decompose it into its connected components $N_1,\ldots,N_n$; set also $K_i := K \cap N_i$. It follows from the observation above that each $\chi h(K_i) \leq 0$ except for those $K_i$ that are attractors or repellers with $\check{H}^1(K_i) = 0$. The latter is equivalent to saying that $K_i$ does not separate $\mathbb{S}^2$. Since by Theorem~\ref{teo:main} we have \[2 = \chi(\mathbb{S}^2) = \sum_{i=1}^n \chi h(K_i),\] at least two summands must be equal to $1$ and so at least two $K_i$ must be nonseparating attractors or repellers. The same holds, then, for any connected component of these $K_i$. 

Finally, suppose $K$ has exactly two connected components $K_1$ and $K_2$. As we have seen, each of them is an attractor or a repeller. Observe first that $K_1$ and $K_2$ cannot be both attractors (or repellers). For, suppose they were. Clearly their basins of attraction are disjoint, and they must cover all of $\mathbb{S}^2$ since otherwise the complement of their union would be a nonempty, compact, invariant subset of $\mathbb{S}^2$ and neither the $\omega$- nor the $\omega^*$-limit of their points would intersect $K$. Thus $\mathbb{S}^2$ would be the union of two disjoint, open, nonempty sets which is a contradiction. Hence if (say) $K_1$ is an attractor, $K_2$ must be a repeller. For any point $x \not\in K_1$ its $\omega^*$-limit is compact and disjoint from $\mathcal{A}(K_1)$ so in particular also from $K_1$. Since $K$ is a $W$-set, $\omega^*(x) \cap K_2 \neq \emptyset$, and this implies that the negative semiorbit of $x$ must visit the repulsion region of $K_2$. The latter is invariant, so $x$ actually belongs to it. Thus in fact $\omega^*(x) \subset K_2$. Summing up, $\{K_1,K_2\}$ is an attractor-repeller decomposition of $\mathbb{S}^2$.
\end{proof}

A nice consequence of Theorem~\ref{teo:positive} is the following result that improves \cite[Theorem~35]{BSdis}.

\begin{corollary}
   Let $K$ be an isolated, connected, global $W$-set for a flow on $\mathbb{R}^2$. Then, $K$ is either a global attractor or a global repeller. 
\end{corollary}

\begin{proof} Extend the flow to the one-point compactification $\mathbb{R}^2\cup\{\infty\}$ leaving the infinity point fixed and then apply Theorem~\ref{teo:positive} to the the two component, isolated, global $W$-set $K\cup\{\infty\}$.
\end{proof}

If we assume that the global $W$-set $K$ in the previous result is an isolated weak attractor, we obtain the following result that improves \cite[Theorem~18]{MSGS}.

\begin{corollary}
    Let $K$ be an isolated, connected, global weak attractor for a flow on $\mathbb{R}^2$. Then $K$ is a global attractor.
\end{corollary}

\section{Non-global $W$-sets} \label{sec:Wsets2}

Suppose that $K$ is a $W$-set of a flow defined on a compact manifold $M$. Unlike the previous section, we now assume that $K$ is not necessarily global and let $\mathcal{W}$ be its witness region. 

\begin{proposition}
$\mathcal{W}$ is an invariant open subset of $M$ with finitely many connected components.
\end{proposition}

\begin{proof}
The invariance of $\mathcal{W}$ is clear since the $\omega$- and $\omega^*$- limit sets coincide for all points of the same trajectory. To see that $\mathcal{W}$ is open it is sufficient to check that if $x\in\partial\mathcal{W}$, then $x\notin\mathcal{W}$. Since $\mathcal{W}$ is invariant so is $\partial \mathcal{W}$ and, hence, if $x\in\partial\mathcal{W}$ the set $\omega(x)\cup\alpha(x)$ must be contained in $\partial\mathcal{W}$ which is disjoint from $K$ because $\mathcal{W}$ is a neighborhood of $K$. Therefore $\mathcal{W}$ is an open set.

The components of $\mathcal{W}$ are open in $M$ and cover $K$, which is compact, so only finitely many of them $\mathcal{W}_1,\ldots,\mathcal{W}_r$ have a nonempty intersection with $K$. Let $C$ be any component of $\mathcal{W}$ and $x \in C$. Since $C$ is invariant, $\omega(x) \cup \omega^*(x) \subset \overline{C}$ and, since $x$ is witnessed by $K$, it follows that $\overline{C} \cap K \neq \emptyset$. In particular $C$ must intersect $\cup_{i=1}^r \mathcal{W}_i$ because the latter is a neighbourhood of $K$, and so $C$ must be one of the $\mathcal{W}_i$. Thus the $\mathcal{W}_i$ are all the components of $\mathcal{W}$.
\end{proof}

In fact the proof of the proposition shows that each $K_i := K \cap \mathcal{W}_i$ is a $W$-set whose witness region is precisely the connected component $\mathcal{W}_i$ of $\mathcal{W}$.

\begin{corollary}\label{coro:dual}
If $K$ is a non-global $W$-set and $M$ is compact then $L=M\setminus\mathcal{W}$ is non-empty and it is the maximal compact invariant set in $M\setminus K$ and, hence, it is isolated as an invariant set. Moreover, $K\cup L$ is a global $W$-set.
\end{corollary}

\begin{proof} It is clear that $L$ is the maximal compact invariant set in $M\setminus K$ since every point $x\notin L$ is witnessed by $K$ and, hence, the closure of its trajectory is not contained in $M\setminus K$. The fact that $K\cup L$ is a global $W$-set is straightforward since every point in $M$ is either witnessed by $K$ or it is contained in the compact invariant set $L$ and, in particular, witnessed by $L$.
\end{proof}

We shall refer to the isolated invariant set $L$ from Corollary~\ref{coro:dual} as the \emph{complementary} isolated invariant set to the $W$-set $K$. Notice that in general $L$ is not a $W$-set itself. For instance, for the flow depicted in Figure~\ref{fig:dis} the isolated non-saddle $K$ is a $W$-set whose complementary isolated invariant set is the saddle fixed point $p$ which is clearly not a $W$-set.

The following useful lemma studies conditions that ensure that the union of a $W$-set with its complementary isolated invariant set is a global non-saddle set when the phase space is a closed $2$-manifold.

\begin{theorem}\label{teo:nondual}
   Suppose that $K$ is an isolated $W$-set and let $L$ be its complementary isolated invariant set. Then $\chi(K \cup L) \geq \chi(\mathcal{W})$ and the equality holds if and only $K \cup L$ is a non-saddle set.
\end{theorem}
\begin{proof} (1) We first prove the theorem in the case when $K$ is already a global $W$-set (i.e. $\mathcal{W} = M$) and so $L = \emptyset$). Let $N$ be an isolating block for $K$. After discarding components of $N$ we may assume that each component of $N$ intersects $K$. Also, recall from the discussion about the Conley index in two dimensions in Section \ref{sec:background} that we may take $N$ to be a compact $2$-manifold, with $N^i$ and $N^o$ compact $1$-submanifolds (with boundary) of $\partial N$.

Since the trajectory of every point in $N-K$ must intersect $\partial N$ either in the future or the past, every connected component of $N - K$ must intersect $\partial N$. Thus $H_0(N-K,\partial N)=0$ and by Alexander duality $\check{H}^2(N,K) = 0$. Thus $\chi(N,K) = -{\rm rk}\ \check{H}^1(N,K) \leq 0$.

Observe that $\chi(M,N) = \chi(E,\partial E) = \chi(E)$ by excision and Lefschetz duality. While proving Theorem \ref{teo:main} we obtained that $\chi(E,N^i) -t$, where $t$ is the number of tangency points in $N$. Hence \[\chi(E) = \chi(E, N^i) + \chi(N^i) = -t + t/2 = -t/2,\] where we have used that $\chi(N^i)$ is the number of contractible components of $N^i$ (i.e. arcs) and each of these is determined by a pair of tangency points.

Putting together the two preceding paragraphs we now have \[\chi(M,K) = \chi(M,N) + \chi(N,K) \leq \chi(E,\partial E) = -t/2 \leq 0.\] Moreover, if the equality $\chi(M,K) = 0$ holds then $t = 0$ so every component of $\partial N$ is comprised entirely of entry or exit points. Thus $N = N^+ \cup N^-$ and $K$ is nonsaddle.

(2) Now we extend the result to the general case. Since $K$ is isolated, its first Betti number is finite and so $\chi(K)$ is either finite (if $K$ has finitely many connected components) or $+\infty$ (if not). The same holds for $L$. Taking into account that $L = M \setminus \mathcal{W}$ and so by Alexander duality $\chi(\mathcal{W}) = \chi(M,L) = \chi(M) - \chi(L)$, we see that $\chi(\mathcal{W})$ is finite or $-\infty$. Thus it makes sense to write \[\chi(\mathcal{W}) - \chi(K) = \chi(M) - \chi(L) - \chi(K) = \chi(M) - \chi(K \cup L)\] even if some of the Euler characteristics are infinite. Since $K\cup L$ is an isolated global $W$-set, case (1) above ensures that $K\cup L$ is non-saddle if and only if $\chi(K\cup L)=\chi(M)$, which by the above amounts to saying $\chi(\mathcal{W}) = \chi(K)$.
\end{proof}

Using Theorem~\ref{teo:nondual} it is possible give a sufficient condition for an  isolated invariant compactum to be  non-saddle for flows on the plane. This result improves and generalizes one of the implications of \cite[Theorem~9]{BSbif}.

\begin{corollary}\label{coro:carac}
Let $K$ be an isolated $W$-set for a flow on a (not necessarily compact) $2$-manifold $M$. Assume that $\chi(K) = \chi(\mathcal{W})$. Then $K$ is non-saddle.
\end{corollary}
\begin{proof} Since $\mathcal{W}$ has finitely many connected components, $\chi(\mathcal{W})$ is either finite or $-\infty$. An isolated invariant set always has $\chi(K)$ finite or $+\infty$, so the assumption $\chi(K) = \chi(\mathcal{W})$ forces both Euler characteristics to be finite. In particular, by the classification of open surfaces \cite[Theorem 3]{Ric}, $\mathcal{W}$ is (homeomorphic to) a closed surface $\mathcal{W}_{\infty}$ with a set $L$ of finitely many points removed. The flow on $\mathcal{W}$ can be extended to $\mathcal{W}_{\infty}$ by leaving the points in $L$ fixed. It then follows from Theorem \ref{teo:nondual} that $K \cup L$ is non-saddle, and in particular so is $K$. 
\end{proof}

\begin{remark}
    The converse of Corollary \ref{coro:carac} also holds if $K$ is connected and $M=\mathbb{R}^2$ or $\mathbb{S}^2$ by \cite[Theorem 9]{BSbif}. However if $K$ has more than one connected component then $\chi(K)=\chi(\mathcal{I}(K))$ does not hold in general (recall that for a non-saddle set $\mathcal{W} = \mathcal{I}(K)$). To see this consider the $2$-sphere $\mathbb{S}^2$ and fix an open meridian $m_0$, that is a meridian with the poles removed. Consider a flow which leaves the poles fixed and every open meridian $m\neq m_0$ is a trajectory connecting the two poles oriented from north pole to south pole, i.e., in $\mathbb{S}^2\setminus m_0$ the flow is topologically equivalent to a north pole south pole dynamics. The open meridian $m_0$ is broken into three trajectories: a fixed point $p$ on the equator, and two trajectories connecting the point $p$ with the poles oriented accordingly. By removing the point $p$ we obtain a flow in $\mathbb{R}^2$ having a global non-saddle set $K$ consisting of two fixed points. In this case $\chi(K)=2$ while $\chi(\mathcal{I}(K))=\chi(\mathbb{R}^2)=1$.
\end{remark}

The following result characterizes when a connected isolated weak attractor is an attractor, improving and encompassing \cite[Corollary~25]{SGTrans}.

\begin{corollary}
Let $K$ be a connected isolated weak attractor for a flow on a (not necessarily compact) $2$-manifold $M$. Then $K$ is an attractor if and only if $\chi(K)=\chi(\mathcal{A}_w(K))$.
\end{corollary}

\begin{proof}
    It $K$ is an attractor then $\mathcal{A}_w(K)=\mathcal{A}(K)$ and, using the flow as in the proof of Theorem~\ref{thm:decomp}, $\chi(K)=\chi(\mathcal{A}(K))$.

    Conversely, assume that $\chi(K) = \chi(\mathcal{A}_w(K))$. Then by Corollary~\ref{coro:carac} $K$ is non-saddle. Let $N$ be an isolating block of the form $N=N^+\cup N^-$ contained in $\mathcal{A}_w(K)$.  Let $x\in N^-$, then, $\omega(x)\cap K\neq\emptyset$ and, hence, $\omega(x)\cup\omega^*(x)\subset K$ but by \cite[Theorem~25]{BSdis} this only can happen if $x\in K$. Therefore, $N=N^+$ and $K$ is an attractor. 
\end{proof}

%Attractors are characterized by the fact that they contain the $J^+$ of every point contained in their basin of attraction. Hence, if $K$ is an isolated unstable attractor, there exists some $x\in\mathcal{A}(K)$ such that $J^+(x)\nsubseteq K$. These points are called \emph{explosion points}. An explosion point $x$ is said to be an \emph{external explosion point} if $x\notin K$. Otherwise $x$ is said to be an \emph{internal explosion point}.

Suppose $K$ is a stable attractor. For any $x \in \mathcal{A}(K)$ and any neighbourhood $V$ of $K$ there exists $T > 0$ such that $x [T,+\infty) \subset V$ because $K$ attracts $x$, but in fact this happens uniformly: there exists a neighbourhood $U$ of $x$ such that $U [T,+\infty) \subset V$. When $K$ is an unstable attractor this uniform attraction is no longer true at every point; those where it fails are called \emph{explosion points}. An explosion point $x$ is said to be an \emph{external explosion point} if $x\notin K$. Otherwise $x$ is said to be an \emph{internal explosion point}. These notions were introduced in \cite{AthPac}. An example of an unstable attractor with external explosions is provided by the Mendelson flow from Example~\ref{ex:isoatr}. The point $q$ is an isolated unstable attractor and if $q=(q_1,0)$, every point in $(q_1,\infty)\times\{0\}$ is an external explosion point. Notice that in this case $\{q\}$ is a saddle set. 

It turns out that the existence of external explosion points is intimately related to the non-saddleness of the unstable attractor. In fact, \cite[Proposition~9]{BSdis} ensures that $K$ has only internal explosions if and only if $K$ is non-saddle. A direct consequence of Theorem~\ref{teo:nondual} is that, for flows on $2$-manifolds, in order to determine the existence of external explosions we only have to compare the Euler characteristic of $K$ with that of its basin of attraction. This version of the result extends \cite[Theorem~15]{SGTrans} to the non-connected case:

\begin{corollary}
   Let $K$ be an isolated unstable attractor in a (not necessarily compact) $2$-manifold $M$. Then $K$ has only internal explosions if and only if $\chi(K)=\chi(\mathcal{A}(K))$.
\end{corollary}

\begin{proof} By \cite[Theorem 1.25, p. 64]{BhSz} any unstable attractor $K$ can be enlarged to a (stable) attractor $\hat{K}$ whose basin of attraction is still the same, $\mathcal{A}(K)$. We call this $\hat{K}$ the stabilization of $K$. By suitably taking a level set of a Lyapunov function one obtains a global section $\Sigma$ of the flow in $\mathcal{A}(K) \setminus \hat{K}$. This means that the flow provides a homeomorphism $\Sigma \times (-\infty,+\infty) \cong \mathcal{A}(K) \setminus \hat{K}$. Since the phase space is a $2$-manifold $\Sigma$ is a closed $1$-manifold \cite{chewningowen1}; i.e. it is a finite union of copies of $\mathbb{S}^1$. We can use this to construct a manifold compactification $\mathcal{A}_{\infty}$ of $\mathcal{A}(K)$ as follows. First take the disjoint union of $\Sigma \times [-\infty,0]$ and $\mathcal{A}(K)$ and identify every $(x,s) \in \Sigma \times (-\infty,0)$ with $xs \in \mathcal{A}(K)$. The result is a manifold with boundary $\Sigma \times \{-\infty\}$ and whose interior is (homeomorphic to) $\mathcal{A}(K)$. By collapsing each boundary component separately to a single point one obtains a closed $2$-manifold $\mathcal{A}_{\infty}$ with a finite set $L$ of distinguished points the complement of which is $\mathcal{A}(K)$. The flow can be extended to $\mathcal{A}_{\infty}$ by leaving the distinguished points fixed.

Recall that $K$ has only internal explosions if and only if it is non-saddle. Notice that $L$ is a repeller (it is the repeller dual to $\hat{K}$), and so $K$ is non-saddle if and only if $K \cup L$ is non-saddle. By Theorem~\ref{teo:nondual}, the latter holds if and only if $\chi(K) = \chi(\mathcal{W}) = \chi(\mathcal{A}(K))$.
\end{proof}

Suppose now that $K$ is an isolated non-saddle set. The structure of the flow in its region of influence is easy to understand in the absence of the so-called dissonant points as we have described briefly in Section~\ref{sec:nonsaddle1}. A \emph{dissonant point} of $K$ is a point $x\in M\setminus K$ that is a limit of homoclinic points $x_n$ but is not itself homoclinic to $K$, i.e., such that either $\omega(x)\cap K=\emptyset$ or $\omega^*(x)\cap K=\emptyset$. If there exist dissonant points $x$ outside the region of influence of $K$, then there also exist dissonant points in the region of influence. These are constructed by taking a sequence of homoclinic points $x_n$ that converges to $x$ and intersecting their orbits with the exit set of an isolating block $N$ of the form $N=N^+\cup N^-$ for $K$ to obtain a sequence of points $y_n$ which has an accumulation point $y$ whose $\omega^*$-limit is contained in $K$ and whose $\omega$-limit does not intersect $K$. This point $y$ therefore belongs to $\mathcal{I}(K) \setminus K$ and is dissonant. Details can be found in \cite[Proposition 20]{BSdis}.

The following result states that, in $2$-manifolds, to determine the existence of dissonant points it is sufficient to compare the Euler characteristics of $K$ and its region of influence. This result extends \cite[Theorem~32]{BSdis} removing assumptions about connectedness of $K$, and compactness and orientability of $M$.

\begin{corollary}
Let $K$ be an isolated non-saddle set in a (not necessarily compact) $2$-manifold $M$. Then $K$ has no dissonant points if and only if $\chi(K)=\chi(\mathcal{I}(K))$.
\end{corollary}

\begin{proof} Suppose first that $K$ has no dissonant points. Then, and referring to the description of $\mathcal{I}(K) \setminus K$ given in Section \ref{sec:nonsaddle1}, every component of $\mathcal{I}(K) \setminus K$ is either homoclinic or $K$ behaves as an attractor or a repeller in it. Homoclinic components already have a compact closure, and we may compactify $\mathcal{I}(K)$ by repeating the construction of the preceding corollary in the attracting or repelling components of $\mathcal{I}(K) \setminus K$. Thus we again obtain a closed $2$-manifold $\mathcal{I}_{\infty}$ with a finite set of fixed points $L$ whose complement is $\mathcal{I}(K)$. Each point in $L$ is either an attractor or a repeller and $K$ is non-saddle, so $K \cup L$ is non-saddle and so $\chi(K) = \chi(\mathcal{W}) = \chi(\mathcal{I}(K))$ by Theorem~\ref{teo:nondual}.

Now assume that $\chi(K) = \chi(\mathcal{I}(K))$. The proof of Corollary \ref{coro:carac} shows that $\mathcal{I}(K)$ can be compactified to a closed $2$-manifold $\mathcal{I}_{\infty}$ by adding a finite set of rest points $L$, and that $K\cup L$ is a global non-saddle set for $\mathcal{I}_{\infty}$. As a consequence, every component of $\mathcal{I}_{\infty}\setminus (K\cup L)=\mathcal{I}(K)\setminus K$ is homoclinic to $K\cup L$ and, hence, parallelizable. Therefore, $K$ does not have dissonant points in $\mathcal{I}(K)\setminus K$ and so it does not have dissonant points at all (as a subset of $M$).
\end{proof}

\bibliographystyle{amsplain}
%    Insert the bibliography data here.
\bibliography{Biblio1}
\end{document}